\documentclass[final,leqno]{siamltex704}
\usepackage{amsmath}
\usepackage{amssymb}
\usepackage{tikz}
\usepackage{graphicx}
\usepackage[notcite,notref]{showkeys}
\newcommand{\bv}{{\bf v}}
\newcommand{\bu}{{\bf u}}
\newcommand{\be}{{\bf e}}
\newcommand{\bn}{{\bf n}}
\newcommand{\bx}{{\bf x}}
\newcommand{\by}{{\bf y}}
\newcommand{\bQ}{{\bf Q}}
\newcommand{\sgn}{\text{sgn}}
\newcommand{\pT}{{\partial T}}
\def\bbQ{\mathbb{Q}}

\def\T{{\mathcal T}}
\def\E{{\mathcal E}}

\def\l{{\langle}}
\def\r{{\rangle}}
\def\3bar{{|\hspace{-.02in}|\hspace{-.02in}|}}
\newtheorem{defi}{Definition}[section]

\newtheorem{algorithm}{Weak Galerkin Algorithm}
\renewcommand{\ldots}{\dotsc}
\title{Weak Galerkin Finite Element
Methods for the Biharmonic Equation on Polytopal Meshes}

\author{Lin Mu\thanks{Department of Mathematics, University of Arkansas at
Little Rock, Little Rock, AR 72204} \and Junping
Wang\thanks{Division of Mathematical Sciences, National Science
Foundation, Arlington, VA 22230 (jwang@\break nsf.gov). The research
of Wang was supported by the NSF IR/D program, while working at
National Science Foundation. However, any opinion, finding, and
conclusions or recommendations expressed in this material are those
of the author and do not necessarily reflect the views of the
National Science Foundation,} \and Xiu Ye\thanks{Department of
Mathematics, University of Arkansas at Little Rock, Little Rock, AR
72204 (xxye@ualr.edu). This research was supported in part by
National Science Foundation Grant DMS-1115097.}}

\begin{document}

\maketitle

\begin{abstract}
A new weak Galerkin (WG) finite element method is introduced and
analyzed in this paper for the biharmonic equation in its primary
form. This method is highly robust and flexible in the element
construction by using discontinuous piecewise polynomials on general
finite element partitions consisting of polygons or polyhedra of
arbitrary shape. The resulting WG finite element formulation is
symmetric, positive definite, and parameter-free. Optimal order
error estimates in a discrete $H^2$ norm is established for the
corresponding WG finite element solutions. Error estimates in the
usual $L^2$ norm are also derived, yielding a sub-optimal order of
convergence for the lowest order element and an optimal order of
convergence for all high order of elements. Numerical results are
presented to confirm the theory of convergence under suitable
regularity assumptions.
\end{abstract}


\pagestyle{myheadings}

\section{Introduction}

This paper is concerned with numerical methods for the biharmonic
equation with boundary conditions. For simplicity, consider the
model problem that seeks an unknown function $u=u(x)$ satisfying
\begin{eqnarray}
\Delta^2 u&=&f,\quad \mbox{in}\;\Omega,\label{pde}\\
u&=&\zeta,\quad\mbox{on}\;\partial\Omega,\label{bc-d}\\
\frac{\partial u}{\partial
n}&=&\phi\quad\mbox{on}\;\partial\Omega,\label{bc-n}
\end{eqnarray}
\medskip
where $\Delta$ is the Laplacian operator, $\zeta=\zeta(x)$ and
$\phi=\phi(x)$ are given functions defined on the boundary of the
domain $\Omega$. Assume that the domain $\Omega$ is open bounded
with a Lipschitz continuous boundary $\partial \Omega$ in
$\mathbb{R}^d, d=2,3$.

A natural variational formulation for the biharmonic equation
(\ref{pde}) with Dirichlet and Neumann boundary conditions
(\ref{bc-d}) and (\ref{bc-n}) seeks $u\in H^2(\Omega)$ satisfying
$u|_{\partial \Omega}=\zeta$ and $\frac{\partial u}{\partial
n}|_{\partial \Omega}=\phi$ such that
\begin{equation}\label{wf}
(\Delta u, \Delta v) = (f,v),\qquad \forall v\in H_0^2(\Omega),
\end{equation}
where $H_0^2(\Omega)$ is a subspace of the Sobolev space
$H^2(\Omega)$ consisting of functions with vanishing value and
normal derivative on $\partial\Omega$. Based on the variational form
(\ref{wf}), one may design various conforming finite element schemes
for (\ref{pde})-(\ref{bc-n}) by constructing finite element spaces
as subspaces of $H^2(\Omega)$. It is known that $H^2$-conforming
methods essentially require $C^1$-continuous piecewise polynomials
on a prescribed finite element partition, which imposes an enormous
difficulty in practical computation. Due to the complexity in the
construction of $C^1$-continuous elements, $H^2$-conforming finite
element methods are rarely used in practice for solving the
biharmonic equation.

As an alternative approach, nonconforming and discontinuous Galerkin
finite element methods have been developed for solving the
biharmonic equation over the last several decades. The Morley
element \cite{morley} is a well-known example of nonconforming
element for the biharmonic equation by using piecewise quadratic
polynomials. Recently, a $C^0$ interior penalty method was studied
in \cite{bs,ghlmt}. In \cite{ms}, a $hp$-version interior-penalty
discontinuous Galerkin method was developed for the biharmonic
equation. To avoid the use of $C^1$-elements, mixed methods have
been developed for the biharmonic equation by reducing the fourth
order problem to a system of two second order equations
\cite{ab,falk,gnp,monk,mwwy}.

In this paper, we will develop a highly flexible and robust weak
Galerkin finite element method for the biharmonic equation that
allows the use of generalized approximating functions on general
partitions consisting of polygons or polyhedra of arbitrary shape
with certain shape regularity. The weak Galerkin method refers to a
numerical technique for partial differential equations where
differential operators are interpreted and approximated as
distributions over a set of generalized functions. The method/idea
was first introduced in \cite{wy} for second order elliptic
equations, and the concept was further developed in \cite{wy-mixed,
wy-stokes, mwy-polytope}. By design, weak Galerkin methods use
generalized and/or discontinuous approximating functions on general
meshes to overcome the barrier in the construction of ``smooth''
finite element functions.

Intuitively, a weak Galerkin finite element scheme for the
biharmonic equation (\ref{pde})-(\ref{bc-n}) can be derived by
replacing the differential operator $\Delta$ in (\ref{wf}) by a
discrete weak Laplacian, denoted by $\Delta_w$. However, such a
straight forward replacement may not work without including a
mechanism that enforces a certain weak continuity of the underlying
approximating functions. A weak continuity can be realized by
introducing an appropriately defined stabilizer, denoted as
$s(\cdot,\cdot)$. Formally, our weak Galerkin finite element method
for (\ref{pde})-(\ref{bc-n}) can be described by seeking a finite
element function $u_h$ satisfying
\begin{equation}\label{wf1}
(\Delta_w u_h, \Delta_w v)+s(u_h,v) = (f, v)
\end{equation}
for all testing functions $v$. The goal of the paper is to specify
all the details for (\ref{wf1}), and further justify the
rigorousness of the method by establishing a mathematical
convergence theory.

The paper is organized as follows. In Section
\ref{Section:preliminaries}, we introduce some standard notations
for Sobolev spaces. A weak Laplacian operator and its discrete
version will be introduced in Section \ref{Section:weak-Laplacian}.
In Section \ref{Section:wg-fem}, we shall present two WG finite
element schemes for the biharmonic equation
(\ref{pde})-(\ref{bc-n}). In Section \ref{Section:L2projections}, we
shall introduce some local $L^2$ projection operators and then
derive some approximation properties which are useful in a
convergence analysis. In Section \ref{Section:error-analysis}, we
shall establish optimal order error estimates for the WG finite
element approximation in a $H^2$-equivalent discrete norm. In
Section \ref{Section:error-estimate-in-L2}, we shall derive an error
estimate for the WG-FEM approximation in the usual $L^2$ norm.
Results from two numerical experiments are reported in Section
\ref{Section:numerical-results}. Finally, we provide some technical
results in the appendix that are critical in dealing with finite
element functions on arbitrary polygons/polyhedra.

\section{Preliminaries and Notations}\label{Section:preliminaries}

Let $D$ be any open bounded domain with Lipschitz continuous
boundary in $\mathbb{R}^d, d=2, 3$. We use the standard definition
for the Sobolev space $H^s(D)$ and the associated inner product
$(\cdot,\cdot)_{s,D}$, norm $\|\cdot\|_{s,D}$, and seminorm
$|\cdot|_{s,D}$ for any $s\ge 0$. For example, for any integer $s\ge
0$, the seminorm $|\cdot|_{s, D}$ is given by
$$
|v|_{s, D} = \left( \sum_{|\alpha|=s} \int_D |\partial^\alpha v|^2
dD \right)^{\frac12}
$$
with the usual notation
$$
\alpha=(\alpha_1, \ldots, \alpha_d), \quad |\alpha| =
\alpha_1+\ldots+\alpha_d,\quad
\partial^\alpha =\prod_{j=1}^d\partial_{x_j}^{\alpha_j}.
$$
The Sobolev norm $\|\cdot\|_{m,D}$ is given by
$$
\|v\|_{m, D} = \left(\sum_{j=0}^m |v|^2_{j,D} \right)^{\frac12}.
$$

The space $H^0(D)$ coincides with $L^2(D)$, for which the norm and
the inner product are denoted by $\|\cdot \|_{D}$ and
$(\cdot,\cdot)_{D}$, respectively. When $D=\Omega$, we shall drop
the subscript $D$ in the norm and inner product notation.

The space $H({\rm div};D)$ is defined as the set of vector-valued
functions on $D$ which, together with their divergence, are square
integrable; i.e.,
\[
H({\rm div}; D)=\left\{ \bv: \ \bv\in [L^2(D)]^d, \nabla\cdot\bv \in
L^2(D)\right\}.
\]
The norm in $H({\rm div}; D)$ is defined by
$$
\|\bv\|_{H({\rm div}; D)} = \left( \|\bv\|_{D}^2 + \|\nabla
\cdot\bv\|_{D}^2\right)^{\frac12}.
$$

\section{Weak Laplacian and Discrete Weak Laplacian}\label{Section:weak-Laplacian}

For the biharmonic problem (\ref{pde})-(\ref{bc-n}) with the
variational form (\ref{wf}), the principle differential operator is
the Laplacian $\Delta$. Thus, we shall introduce a weak Laplacian
operator defined on a class of discontinuous functions. For
numerical purposes, we will define a discrete version of the weak
Laplacian in polynomial subspaces.

Let $T$ be any polygonal or polyhedral domain with boundary
$\partial T$. By a {\em weak function} on the region $T$ we mean a
function $v=\{v_0, v_b, \bv_{g}\}$ such that $v_0\in L^2(T)$,
$v_b\in H^{\frac12}(\partial T)$, and $\bv_g\cdot\bn\in
H^{-\frac12}(\partial T)$, where $\bn$ is the outward normal
direction of $T$ on its boundary. The first and second  components
$v_0$ and $v_b$ can be understood as the value of $v$ in the
interior and on the boundary of $T$. The third component $\bv_g$
intends to represent the gradient of $v$ on the boundary of $T$.
Note that $v_b$ and $\bv_g$ may not be necessarily related to the
trace of $v_0$ and $\nabla v_0$ on $\partial T$, respectively.

Denote by $W(T)$ the space of all weak functions on $T$; i.e.,
\begin{equation}\label{www}
W(T) = \{v=\{v_0, v_b,\bv_g \}:\ v_0\in L^2(T),\; v_b\in
H^{\frac12}(\partial T), \ \bv_g\cdot\bn\in H^{-\frac12}(\partial
T)\}.
\end{equation}
Let $\langle\cdot,\cdot\rangle_\pT$ be the inner product in
$L^2(\pT)$. Define $G_2(T)$ by
$$
G_2(T)=\{\varphi: \ \varphi\in H^1(T), \Delta\varphi\in L^2(T)\}.
$$
For any $\varphi\in G_2(T)$, we have $\nabla\varphi\in H(div,T)$,
and hence $\nabla\varphi\cdot\bn\in H^{-\frac12}(\partial T)$.

\medskip

\begin{defi}
The dual of $L^2(T)$ can be identified with itself by using the
standard $L^2$ inner product as the action of linear functionals.
With a similar interpretation, for any $v\in W(T)$, the {\em weak
Laplacian} of $v=\{v_0, v_b,\bv_g \}$ is defined as a linear
functional $\Delta_w v$ in the dual space of $G_2(T)$ whose action
on each $\varphi\in G_2(T)$ is given by
\begin{equation}\label{wl}
(\Delta_w v, \varphi)_T = (v_0, \Delta\varphi)_T-\l
v_b,\nabla\varphi\cdot\bn\r_\pT +\l \bv_g\cdot\bn, \varphi\r_\pT,
\end{equation}
where $\bn$ is the outward normal direction to $\partial T$.
\end{defi}

The Sobolev space $H^2(T)$ can be embedded into the space $W(T)$ by
an inclusion map $i_W: \ H^2(T)\to W(T)$ defined as follows
$$
i_W(\phi) = \{\phi|_{T}, \phi|_{\partial T}, \nabla\phi|_{\partial
T}\},\qquad \phi\in H^2(T).
$$
With the help of the inclusion map $i_W$, the Sobolev space $H^2(T)$
can be viewed as a subspace of $W(T)$ by identifying each $\phi\in
H^2(T)$ with $i_W(\phi)$. Analogously, a weak function
$v=\{v_0,v_b,\bv_g\}\in W(T)$ is said to be in $H^2(T)$ if it can be
identified with a function $\phi\in H^2(T)$ through the above
inclusion map. It is not hard to see that the weak Laplacian is
identical with the strong Laplacian in $H^2(T)$; i.e., $\Delta_w
v=\Delta v$ for all functions $v\in H^2(T)$.

Next, we introduce a discrete weak Laplacian operator by
approximating $\Delta_w$ in a polynomial subspace of the dual of
$G_2(T)$. To this end, for any non-negative integer $r\ge 0$, denote
by $P_{r}(T)$ the set of polynomials on $T$ with degree no more than
$r$. A discrete weak Laplacian operator, denoted by
$\Delta_{w,r,T}$, is defined as the unique polynomial
$\Delta_{w,r,T}v \in P_r(T)$ satisfying the following equation
\begin{equation}\label{dwl}
(\Delta_{w,r,T} v, \varphi)_T = ( v_0, \Delta\varphi)_T-\l
v_b,\nabla\varphi\cdot\bn\r_\pT +\l \bv_g\cdot\bn,
\varphi\r_\pT,\quad \forall \varphi\in P_r(T).
\end{equation}

\section{Weak Galerkin Finite Element Schemes}\label{Section:wg-fem}

Let ${\cal T}_h$ be a partition of the domain $\Omega$ into polygons
in 2D or polyhedra in 3D. Assume that ${\cal T}_h$ is shape regular
in the sense as defined in \cite{wy-mixed} (seel also Appendix
\ref{Section:appendix}). Denote by ${\cal E}_h$ the set of all edges
or flat faces in ${\cal T}_h$, and let ${\cal E}_h^0={\cal
E}_h\backslash\partial\Omega$ be the set of all interior edges or
flat faces.

For any given integer $k\ge 2$, denote by $W_k(T)$ the discrete weak
function space given by
\begin{equation}\label{Wkspace}
W_k(T)=\{\{v_0,v_b, \bv_g\}:\; {v_0}\in P_k(T), \ v_b\in P_k(e), \
\bv_g\in [P_{k-1}(e)]^d,  \ e\subset\partial T\}.
\end{equation}
By patching $W_k(T)$ over all the elements $T\in {\cal T}_h$ through
a common value on the interface ${\cal E}_h^0$, we arrive at a weak
finite element space $V_h$ given by
\[
V_h=\{\{v_0,v_b, \bv_g\}:\ \{v_0,v_b, \bv_g\}|_T\in W_k(T),\ \forall
T\in\T_h\}.
\]

Denote by $\Lambda_h$ the trace of $V_h$ on $\partial\Omega$ from
the component $v_b$. It is clear that $\Lambda_h$ consists of
piecewise polynomials of degree $k$. Similarly, denote by
$\Upsilon_h$ the trace of $V_h$ from the normal component of $\bv_g$
as piecewise polynomials of degree $k-1$. Denote by $V_h^0$ the
subspace of $V_h$ with vanishing traces; i.e.,
$$
V_h^0=\{v=\{v_0,v_b,\bv_{g}\}\in V_h, {v_b}|_e=0,\
{\bv_{g}}\cdot\bn|_e=0,\ e\subset\partial T\cap\partial\Omega\}.
$$

Denote by $\Delta_{w,k-2}$ the discrete weak Laplacian operator on
the finite element space $V_h$ computed by using (\ref{dwl}) on each
element $T$ for $k\ge 2$; i.e.,
$$
(\Delta_{w,k-2}v)|_T =\Delta_{w,k-2,T} (v|_T),\qquad \forall v\in
V_h.
$$
For simplicity, we shall drop the subscript $k-2$ in the notation
$\Delta_{w,k-2}$ for the discrete weak Laplacian. We also introduce
the following notation
$$
(\Delta_w v,\Delta_w w)_h:=\sum_{T\in {\cal T}_h}(\Delta_w
v,\Delta_w w)_T.
$$

\subsection{Algorithm I}

For any $u_h=\{u_0,u_b,\bu_{g}\}$ and $v=\{v_0,v_b,\bv_{g}\}$ in
$V_h$, we introduce a bilinear form as follows
\[
s(u_h, v):=\sum_{T\in\T_h} h_T^{-1}\langle \nabla u_0-\bu_{g},
\nabla v_0-\bv_{g}\rangle_\pT + \sum_{T\in\T_h} h_T^{-3}\langle
u_0-u_{b}, v_0-v_{b}\rangle_\pT.
\]

\smallskip
\begin{algorithm}
A numerical approximation for (\ref{pde})-(\ref{bc-n}) can be
obtained by seeking $u_h=\{u_0,u_b,\bu_{g}\}\in V_h$ satisfying
$u_b=Q_b \zeta$ and $\bu_{g}\cdot\bn=Q_{gn}\phi$ on $\partial
\Omega$ and the following equation:
\begin{equation}\label{wg}
(\Delta_w u_h,\Delta_w v)_h+s(u_h,v)=(f,v_0), \quad\forall\
v=\{v_0,v_b,\bv_{g}\}\in V_h^0,
\end{equation}
where $Q_b \zeta$ is the standard $L^2$ projection onto the trace
space $\Lambda_h$ and $Q_{gn}\phi$ is the $L^2$ projection onto the
normal component of the gradient trace space $\Upsilon_h$.
\end{algorithm}

\smallskip

The following is a useful observation concerning the finite element
space $V_h^0$.

\begin{lemma}\label{3bar-is-a-norm}
For any $v\in V_h^0$, let $\3bar v \3bar$ be given as follows
\begin{equation}\label{3barnorm}
\3bar v\3bar^2=(\Delta_w v,\Delta_w v)_h+s(v,\;v).
\end{equation}
Then, $\3bar\cdot\3bar$ defines a norm in the linear space $V_h^0$.
\end{lemma}

\begin{proof}
We shall only verify the positivity property for $\3bar\cdot\3bar$.
To this end, assume that $\3bar v\3bar=0$ for some
$\{v_0,v_b,\bv_{g}\}\in V_h^0$. It follows that $\Delta_w v=0,\
v_0=v_b,$ and $\nabla v_0=\bv_{g}$ on each element $T$ or $\pT$, as
appropriate. We claim that $\Delta v_0=0$ holds true locally on each
element $T$. To this end, for any $\varphi\in P_{k-2}(T)$ we use
$\Delta_w v=0$ and the definition (\ref{dwl}) to obtain
\begin{eqnarray}
0&=&(\Delta_{w} v, \varphi)_T\\
&=& (v_0, \Delta\varphi)_T-\l v_b,\nabla\varphi\cdot\bn\r_\pT +\l
\bv_g\cdot\bn, \
\varphi\r_\pT\nonumber\\
&=&(\Delta v_0, \varphi)_T+\l v_0-v_b,\nabla\varphi\cdot\bn\r_\pT
+\l \bv_g\cdot\bn-\nabla v_0\cdot\bn,\varphi\r_\pT\nonumber\\
&=&(\Delta v_0, \varphi)_T,\label{april14.01}
\end{eqnarray}
where we have used the fact that $v_0-v_b=0$ and $\nabla
v_0-\bv_g=0$ in the last equality. The identity (\ref{april14.01})
implies that $\Delta v_0=0$ holds true locally on each element $T$.
This, together with $v_0=v_b$ and $\nabla v_0=\bv_{g}$ on $\pT$,
shows that $v$ is a smooth harmonic function globally on $\Omega$.
The boundary condition of $v_b=0$ then implies that $v\equiv 0$ on
$\Omega$, which completes the proof.
\end{proof}

\begin{lemma}
The weak Galerkin finite element scheme (\ref{wg}) has a unique
solution.
\end{lemma}

\begin{proof} Let $\bu_h^{(1)}$ and $\bu_h^{(2)}$ be two solutions
of the weak Galerkin finite element scheme (\ref{wg}). It is clear
that the difference $\be_h=\bu_h^{(1)}-\bu_h^{(2)}$ is a finite
element function in $V_h^0$ satisfying
\begin{equation}\label{wg:extension}
(\Delta_w \be_h,\Delta_w v)_h+s(\be_h,v)=0, \quad\forall\
v=\{v_0,v_b,\bv_{g}\}\in V_h^0.
\end{equation}
By setting $v=\be_h$ in (\ref{wg:extension}) we obtain
\[
(\Delta_w \be_h,\Delta_w \be_h)_h+s(\be_h, \be_h)=0.
\]
It follows from Lemma \ref{3bar-is-a-norm} that $\be_h\equiv 0$.
This shows that $\bu_h^{(1)} = \bu_h^{(2)}$.
\end{proof}

\subsection{Algorithm II}

Here we describe another weak Galerkin finite element scheme that
has less number of unknowns than (\ref{wg}). This second WG-FEM
scheme is formulated in a subspace of $V_h$, denoted by $\tilde
V_h$, that uses only the normal component of the vector $\bv_g$ for
any $v=\{v_0,v_b,\bv_g\} \in V_h$. To be more precise, let us
introduce a set of normal directions on ${\cal E}_h$ as follows
\begin{equation}\label{thanks.101}
{\cal D}_h = \{\bn_e: \mbox{ $\bn_e$ is unit and normal to $e$},\
e\in {\cal E}_h \}.
\end{equation}
The weak Galerkin finite element space $\tilde V_h$ is given as
follows
\begin{equation}\label{tildevh}
\tilde V_h=\{v=\{v_0,v_b, v_{g}\bn_e\}:\ v_0\in P_{k}(T), v_b\in
P_{k}(e), v_{g}\in P_{k-1}(e), e\subset\partial T\},
\end{equation}
where $v_g$ can be viewed as an approximation of $\nabla
v\cdot\bn_e$. Denote by $\tilde V_h^0$ the subspace of $V_h$ with
vanishing traces; i.e.,
$$
\tilde V_h^0=\{v=\{v_0,v_b,v_{g}\bn_e\}\in V_h, {v_b}|_e=0,\
{v_{g}}|_e=0,\ e\subset\partial T\cap\partial\Omega\}.
$$

For any $u_h=\{u_0,u_b,u_{g}\bn_e\}$ and $v=\{v_0,v_b,v_{g}\bn_e\}$
in $\tilde V_h$, we introduce a bilinear form as follows
\[
\tilde s(u_h, v):=\sum_{T\in\T_h} h_T^{-1}\langle \nabla
u_0\cdot\bn_e-u_{g}, \nabla v_0\cdot\bn_e-v_{g}\rangle_\pT +
\sum_{T\in\T_h} h_T^{-3}\langle u_0-u_{b}, v_0-v_{b}\rangle_\pT.
\]

\smallskip
\begin{algorithm}
A numerical approximation for (\ref{pde})-(\ref{bc-n}) can be
obtained by seeking $u_h=\{u_0,u_b,u_{g}\bn_e\}\in \tilde V_h$
satisfying $u_b=Q_b \zeta$ and $u_{g}=Q_{gn}\phi$ on $\partial
\Omega$ and the following equation:
\begin{equation}\label{wg-2}
(\Delta_w u_h,\Delta_w v)_h+\tilde s(u_h,v)=(f,v_0), \quad\forall\
v=\{v_0,v_b,v_{g}\bn_e\}\in \tilde V_h^0,
\end{equation}
where $Q_b \zeta$ is the standard $L^2$ projection onto the trace
space $\Lambda_h$ and $Q_{gn}\phi$ is the $L^2$ projection onto the
normal component of the gradient trace space $\Upsilon_h$.
\end{algorithm}

\smallskip

Like (\ref{wg}), the WG-FEM scheme (\ref{wg-2}) can be proved to
have one and only one solution in the corresponding finite element
space. Details are left to interested readers for a verification.

\section{$L^2$ Projections and Approximation Properties}\label{Section:L2projections}
For each element $T$, denote by $Q_0$ the $L^2$ projection onto
$P_{k}(T), k\ge 2$. For each edge/face $e\subset \partial T$, denote
by $Q_b$ and $\bQ_g$ the $L^2$ projection onto $P_{k}(e)$ and
$[P_{k-1}(e)]^d$, respectively. Now for any $u\in H^2(\Omega)$, we
can define a projection into the finite element space $V_h$ such
that on the element $T$
$$
Q_h u = \{Q_0 u,Q_bu,\bQ_{g} (\nabla u)\}.
$$
In addition, denote by $\bbQ_h$ the local $L^2$ projection onto
$P_{k-2}(T)$.

\begin{lemma}\label{Lemma:commutativity}
On each element $T\in \T_h$, the $L^2$ projections $Q_h$ and
$\bbQ_h$ satisfy the following commutative property with the
Laplacian $\Delta$ and the discrete weak Laplacian $\Delta_w$:
\begin{equation}\label{key-Laplacian}
\Delta_{w} (Q_h u) = \bbQ_h (\Delta u)
\end{equation}
for all $u\in H^2(T)$.
\end{lemma}

\begin{proof}
For any $\tau\in P_{k-2}(T)$, it is not hard to see that
\begin{eqnarray*}
(\Delta_{w} Q_h u,\tau)_T &=& (Q_0 u,\Delta\tau)_T + \langle
\bQ_{g}(\nabla u)\cdot\bn, \tau\rangle_{\pT}-\langle Q_b u,\nabla\tau\cdot\bn \rangle_{\pT}\\
&=&(u, \Delta\tau)_T + \langle \nabla
u\cdot\bn,\tau\rangle_{\partial T}-\langle u,\nabla\tau\cdot\bn
\rangle_{\pT}
\\
&=&(\Delta u,\tau)_T\\
&=&(\bbQ_h\Delta u,\tau)_T,
\end{eqnarray*}
which implies the desired identity (\ref{key-Laplacian}).
\end{proof}

The commutative property (\ref{key-Laplacian}) indicates that the
discrete weak Laplacian of the $L^2$ projection of smooth functions
is a good approximation of the Laplacian of the function itself in
the classical sense. This is a nice and useful property of the
discrete weak Laplacian in application to algorithm design and
analysis.

\medskip

The following lemma provides some approximation properties for the
projection operators $Q_h$ and $\bbQ_h$.

\begin{lemma}\label{l1}
Let $\T_h$ be a finite element partition of $\Omega$ satisfying the
shape regularity assumption as specified in \cite{wy-mixed}. Then,
for any $0\le s\le 2$ and $2\le m \le k$ we have
\begin{eqnarray}
&&\sum_{T\in\T_h} h_T^{2s}\|u-Q_0u\|_{s,T}^2 \le Ch^{2(m+1)}
\|u\|^2_{m+1},\label{Qh}\\
&&\sum_{T\in\T_h} h_T^{2s}\|\Delta u-\bbQ_h\Delta u\|^2_{s,T} \le
Ch^{2(m-1)}\|u\|^2_{m+1}.\label{Rh}
\end{eqnarray}
\end{lemma}

\begin{proof}
The proof of this lemma is similar to that of Lemma 5.1 in
\cite{wy-mixed}, and the details are thus omitted.
\end{proof}

\smallskip

Using the trace inequality (\ref{App:trace}) in the Appendix with
$p=2$, we can derive the following estimates which are useful in the
convergence analysis for the weak Galerkin finite element schemes
(\ref{wg}) and (\ref{wg-2}).

\begin{lemma}\label{l2}
Let $2\le m\le k$, $w\in H^{\max\{m+1,4\}}(\Omega)$, and $v\in V_h$.
There exists a constant $C$ such that the following estimates hold
true:
\begin{eqnarray}
&  \left(\sum_{T\in\T_h} h_T\|\Delta w-\bbQ_h\Delta w\|_{\partial
T}^2\right)^{\frac12}
\leq C h^{m-1}\|w\|_{m+1},\label{mmm1}\\
&  \left(\sum_{T\in\T_h} h_T^3\|\nabla(\Delta w-\bbQ_h\Delta
w)\|_\pT^2\right)^{\frac12} \leq Ch^{m-1}(\|w\|_{m+1}
+h\delta_{m,2}\|w\|_4),
\label{mmm2}\\
&  \left(\sum_{T\in\T_h} h_T^{-1}\| \nabla (Q_0w)- \bQ_g(\nabla
w)\|_\pT^2\right)^{\frac12}\le Ch^{m-1}\|w\|_{m+1},\label{mmm3} \\
&  \left(\sum_{T\in\T_h}
h_T^{-3}\|Q_0w-Q_bw\|_\pT^2\right)^{\frac12}\le
Ch^{m-1}\|w\|_{m+1}.\label{mmm4}
\end{eqnarray}
Here $\delta_{i,j}$ is the usual Kronecker's delta with value $1$
when $i=j$ and value $0$ otherwise.
\end{lemma}

\begin{proof}
To derive (\ref{mmm1}), we use the trace inequality
(\ref{App:trace}) and the estimate (\ref{Rh}) to obtain
\begin{eqnarray*}
&& \sum_{T\in\T_h} h_T\|\Delta w-\bbQ_h\Delta
w\|_\pT^2\\
\le && C\sum_{T\in\T_h} \left(\|\Delta w-\bbQ_h\Delta
w\|_T^2+h_T^2\|\nabla(\Delta w-\bbQ_h\Delta w)
\|_T^2\right)\\
&&\le C h^{2m-2}\|w\|_{m+1}^2.
\end{eqnarray*}

As to (\ref{mmm2}), we use the trace inequality (\ref{App:trace})
and the estimate (\ref{Rh}) to obtain
\begin{eqnarray*}
&& \sum_{T\in\T_h}h_T^3\|\nabla(\Delta w-\bbQ_h\Delta w)\|_\pT^2 \\
\le && C\sum_{T\in\T_h}\left( h_T^2\|\nabla(\Delta w-\bbQ_h\Delta
w)\|_T^2+h_T^4\|\nabla^2(\Delta w-\bbQ_h\Delta w)\|_T^2\right)\\
\le && Ch^{2m-2}(\|w\|_{m+1}^2 +h^2\delta_{m,2}\|w\|_4^2).
\end{eqnarray*}

As to (\ref{mmm3}), we have from the definition of $\bQ_{g}$, the
trace inequality (\ref{App:trace}), and the estimate (\ref{Qh}) that
\begin{eqnarray*}
&& \sum_{T\in\T_h} h_T^{-1}\| \nabla (Q_0w)- \bQ_g(\nabla w)\|_\pT^2 \\
\le && \sum_{T\in\T_h}h_T^{-1}\|\nabla Q_0 w-\nabla w\|_\pT^2\\
\le && C\sum_{T\in\T_h}\left(h_T^{-2}\|\nabla Q_0 w-\nabla w\|_T^2
+ \|\nabla Q_0 w-\nabla w\|_{1,T}^2\right)\\
\le && C h^{2m-2}\|w\|_{m+1}^2.
\end{eqnarray*}

Finally, we use the definition of $Q_b$ and the trace inequality
(\ref{App:trace}) to obtain
\begin{eqnarray*}
&& \sum_{T\in\T_h}h_T^{-3}\|Q_0 w-Q_b w\|_\pT^2
\le \sum_{T\in\T_h}h_T^{-3}\|Q_0 w-w\|_\pT^2\\
\le && C\sum_{T\in\T_h}\left(h_T^{-4}\|Q_0 w-w\|_T^2 +
h_T^{-2}\|\nabla(Q_0 w-w)\|_T^2 \right)\\
\le && C h^{2m-2}\|w\|_{m+1}^2.
\end{eqnarray*}
This completes the proof of (\ref{mmm4}), and hence the lemma.
\end{proof}

\section{An Error Estimate in $H^2$}\label{Section:error-analysis}

The goal here is to establish an error estimate for the WG-FEM
solution $u_h$ arising from (\ref{wg}) and (\ref{wg-2}). For
simplicity, we will focus on the error analysis for (\ref{wg}) only;
the analysis can be easily modified to cover the WG-FEM scheme
(\ref{wg-2}) without any difficulty.

First of all, let us derive an error equation for the weak Galerkin
finite element solution.

\begin{lemma}\label{Lemma:error-equation}
Let $u$ and $u_h=\{u_0,u_b,\bu_g\}\in V_h$ be the solution of
(\ref{pde})-(\ref{bc-n}) and (\ref{wg}), respectively. Denote by
$$
e_h=Q_hu - u_h
$$
the error function between the $L^2$ projection of $u$ and its weak
Galerkin finite element approximation. Then the error function $e_h$
satisfies the following equation
\begin{eqnarray}
(\Delta_{w} e_h, \Delta_w v)_h+s(e_h,v)&=&\sum_{T\in\T_h}\langle
\Delta u-\bbQ_h\Delta u,
(\nabla v_0-\bv_{g})\cdot\bn\rangle_{\pT}\nonumber\\
& -& \sum_{T\in\T_h}\langle\nabla(\Delta u-\bbQ_h\Delta u)\cdot\bn,
v_0-v_b\rangle_{\pT}+s(Q_hu,v)\label{ee}
\end{eqnarray}
for all $v\in V_h^0$.
\end{lemma}

\begin{proof}
Using (\ref{dwl}), the integration by parts, and the fact that
$\Delta_w Q_hu = \bbQ_h(\Delta u)$ we obtain
\begin{eqnarray*}
& & (\Delta_{w} Q_h u, \Delta_w v)_T \\
&=& (v_0, \Delta(\Delta_w Q_h
u))_T + \langle
\bv_{g}\cdot\bn,\Delta_w Q_h u\rangle_{\pT}-\langle v_b, \nabla(\Delta_w Q_h u)\cdot\bn \rangle_{\pT}\nonumber\\
&=&(\Delta v_0,\Delta_w Q_h u)_T+\langle v_0, \nabla(\Delta_w Q_h
u)\cdot\bn\rangle_{\pT}-\langle\nabla v_0\cdot\bn, \Delta_w Q_h
u\rangle_{\pT}\nonumber\\
&& +\langle
\bv_{g}\cdot\bn, \Delta_w Q_h u\rangle_{\pT}-\langle v_b, \nabla(\Delta_w Q_h u)\cdot\bn \rangle_{\pT}\nonumber\\
&=&(\Delta v_0,\Delta_w Q_h u)_T+\langle v_0-v_b, \nabla(\Delta_w
Q_h u)\cdot\bn\rangle_{\pT}-\langle (\nabla v_0-\bv_{g})\cdot\bn,
\Delta_w Q_h
u\rangle_{\pT}\nonumber\\
&=&(\Delta v_0,\bbQ_h\Delta u)_T+\langle v_0-v_b,
\nabla(\bbQ_h\Delta u)\cdot\bn\rangle_{\pT}-\langle (\nabla
v_0-\bv_{g})\cdot\bn,
\bbQ_h\Delta u\rangle_{\pT}\nonumber\\
&=&(\Delta u,\Delta v_0)_T+\langle v_0-v_b, \nabla(\bbQ_h\Delta
u)\cdot\bn\rangle_{\pT} -\langle (\nabla v_0-\bv_{g})\cdot\bn,
\bbQ_h \Delta u\rangle_{\pT},
\end{eqnarray*}
which implies that
\begin{eqnarray}
(\Delta u,\Delta v_0)_T&=&(\Delta_{w} Q_h u, \Delta_w v)_T-\langle v_0-v_b,
\nabla(\bbQ_h \Delta u)\cdot\bn\rangle_{\pT}\nonumber\\
&+&\langle (\nabla v_0-\bv_{g})\cdot\bn, \bbQ_h \Delta
u\rangle_{\pT}.\label{key}
\end{eqnarray}
Next, it follows from the integration by parts that
$$
(\Delta u,\Delta v_0)_T = (\Delta^2u, v_0)_T+\langle \Delta u,
\nabla v_0\cdot\bn\rangle_{\pT} - \langle\nabla(\Delta u)\cdot\bn,
v_0\rangle_{\pT}.
$$
Summing over all $T$ and then using the identity $(\Delta^2u,
v_0)=(f,v_0)$ we arrive at
\begin{eqnarray*}
\sum_{T\in\T_h}(\Delta u,\Delta v_0)_T =(f,v_0)
+\sum_{T\in\T_h}\langle \Delta u, (\nabla
v_0-\bv_{g})\cdot\bn\rangle_{\pT} -
\sum_{T\in\T_h}\langle\nabla(\Delta u)\cdot\bn,
v_0-v_b\rangle_{\pT},
\end{eqnarray*}
where we have used the fact that $\bv_g\cdot\bn$ and $v_b$ vanishes
on the boundary of the domain. Combining the above equation with
(\ref{key}) leads to
\begin{eqnarray*}
(\Delta_{w} Q_h u, \Delta_w v)_h&=&(f,v_0)+\sum_{T\in\T_h}\langle
\Delta u-\bbQ_h\Delta u,
(\nabla v_0-\bv_{g})\cdot\bn\rangle_{\pT}\\
&-& \sum_{T\in\T_h}\langle\nabla(\Delta u-\bbQ_h\Delta u)\cdot\bn,
v_0-v_b\rangle_{\pT}.
\end{eqnarray*}
Adding $s(Q_hu,\ v)$ to both sides of the above equation gives
\begin{eqnarray}
(\Delta_{w} Q_h u, \Delta_w
v)_h+s(Q_hu,v)&=&(f,v_0)+\sum_{T\in\T_h}\langle \Delta
u-\bbQ_h\Delta u,
(\nabla v_0-\bv_{g})\cdot\bn\rangle_{\pT}\nonumber\\
& -& \sum_{T\in\T_h}\langle\nabla(\Delta u-\bbQ_h\Delta u)\cdot\bn,
v_0-v_b\rangle_{\pT}+s(Q_hu,v).\label{qh}
\end{eqnarray}
Subtracting (\ref{wg}) from (\ref{qh}) yields the following error
equation
\begin{eqnarray}
(\Delta_{w} e_h, \Delta_w v)_h+s(e_h,v)&=&\sum_{T\in\T_h}\langle
\Delta u-\bbQ_h\Delta u,
(\nabla v_0-\bv_{g})\cdot\bn\rangle_{\pT}\nonumber\\
& -& \sum_{T\in\T_h}\langle\nabla(\Delta u-\bbQ_h\Delta u)\cdot\bn,
v_0-v_b\rangle_{\pT}+s(Q_hu,v)\nonumber
\end{eqnarray}
for all $v\in V_h^0$. This completes the derivation of (\ref{ee}).
\end{proof}

The following is an estimate for the error function $e_h$ in the
trip-bar norm which is essentially an $H^2$-norm in $V_h^0$.

\begin{theorem} Let $u_h\in V_h$  be the weak Galerkin finite element solution arising from
(\ref{wg}) with finite element functions of order $k\ge 2$. Assume
that the exact solution of (\ref{pde})-(\ref{bc-n} ) is sufficiently
regular such that $u\in H^{\max\{k+1,4\}}(\Omega)$. Then, there
exists a constant $C$ such that
\begin{equation}\label{err1}
\3bar u_h-Q_hu\3bar \le Ch^{k-1}
\left(\|u\|_{k+1}+h\delta_{k,2}\|u\|_{4}\right).
\end{equation}
In other words, we have the optimal order of convergence in the
$H^2$ norm.
\end{theorem}
\begin{proof}
By letting $v=e_h$ in the error equation (\ref{ee}), we obtain the
following identity
\begin{eqnarray}\label{March02:00}
\3bar e_h\3bar&=&\sum_{T\in\T_h}\langle \Delta u-\bbQ_h\Delta u,
(\nabla e_0-\be_{g})\cdot\bn\rangle_{\pT}\\
& -& \sum_{T\in\T_h}\langle\nabla(\Delta u-\bbQ_h\Delta u)\cdot\bn,
e_0-e_b\rangle_{\pT}\nonumber\\
&+&\sum_{T\in\T_h} h_T^{-1}\langle \nabla Q_0u-\bQ_{g}\nabla u, \nabla e_0-\be_{g}
\rangle_\pT\nonumber\\
&+& \sum_{T\in\T_h} h_T^{-3}\langle Q_0u-Q_bu,
e_0-e_{b}\rangle_\pT.\nonumber
\end{eqnarray}
Now using the Cauchy-Schwarz inequality and the estimates
(\ref{mmm1}) and (\ref{mmm2}) of Lemma \ref{l2} with $m=k$ one
arrives at
\begin{eqnarray}
&&\left|\sum_{T\in\T_h} \langle \Delta u-\bbQ_h\Delta u, (\nabla
e_0-\be_{g})\cdot\bn\rangle_\pT\right|\label{March02:01}\\
\le &&\left(\sum_{T\in\T_h} h_T\|\Delta u-\bbQ_h\Delta
u\|_\pT^2\right)^{\frac12} \left(\sum_{T\in\T_h} h_T^{-1}
\|\nabla e_0-\be_{g}\|_\pT^2\right)^{\frac12}\nonumber\\
\le && C h^{k-1}\|u\|_{k+1} \3bar e_h\3bar\nonumber
\end{eqnarray}
and
\begin{eqnarray}
&&\left|\sum_{T\in\T_h}\langle \nabla(\Delta u-\bbQ_h\Delta
u)\cdot\bn, e_0-e_b\rangle_\pT\right| \label{March02:02}\\
\le && \left(\sum_{T\in\T_h}h_T^3\|\nabla(\Delta u-\bbQ_h\Delta
u)\|_\pT^2\right)^{\frac12}
\left(\sum_{T\in\T_h}h_T^{-3}\|e_0-e_b\|^2_{\pT}\right)^{\frac12}\nonumber\\
\le && C h^{k-1}(\|u\|_{k+1} + h\delta_{k,2} \|u\|_4) \3bar
e_h\3bar.\nonumber
\end{eqnarray}
Similarly, it follows from the Cauchy-Schwarz and the estimates
(\ref{mmm3}) and (\ref{mmm4}) that
\begin{eqnarray}\label{March02:03}
\left|\sum_{T\in\T_h} h_T^{-1}\langle \nabla Q_0u-\bQ_{g}\nabla u,
\nabla e_0-\be_{g}\rangle_\pT\right|\le C h^{k-1}\|u\|_{k+1} \3bar
e_h\3bar
\end{eqnarray}
and
\begin{eqnarray}\label{March02:04}
\left|\sum_{T\in\T_h} h_T^{-3}\langle Q_0u-Q_bu,\;
e_0-e_b\rangle_\pT\right| \le C h^{k-1}\|u\|_{k+1} \3bar e_h\3bar.
\end{eqnarray}
Substituting (\ref{March02:01})-(\ref{March02:04}) into
(\ref{March02:00}) yields
\[
\3bar e_h\3bar^2 \le
Ch^{k-1}\left(\|u\|_{k+1}+h\delta_{k,2}\|u\|_{4}\right)\3bar
e_h\3bar,
\]
which implies (\ref{err1}). This completes the proof of the theorem.
\end{proof}

\section{An Error Estimate in
$L^2$}\label{Section:error-estimate-in-L2}

This section shall establish an estimate for the first component of
the error function $e_h$ in the standard $L^2$ norm. To this end, we
consider the following dual problem
\begin{eqnarray}
\Delta^2w&=& e_0\quad
\mbox{in}\;\Omega,\label{dual}\\
w&=&0,\quad\mbox{on}\;\partial\Omega,\label{dual1}\\
\nabla w\cdot\bn&=&0\quad\mbox{on}\;\partial\Omega.\label{dual2}
\end{eqnarray}
Assume that the above dual problem has the following regularity
estimate
\begin{equation}\label{reg}
\|w\|_4\le C\|e_0\|.
\end{equation}

\begin{theorem}
Let $u_h\in V_h$ be the weak Galerkin finite element solution
arising from (\ref{wg}) with finite element functions of order $k\ge
2$. Let $t_0=\min(k,3)$. Assume that the exact solution of
(\ref{pde})-(\ref{bc-n} ) is sufficiently regular such that $u\in
H^{k+1}(\Omega)$ and the dual problem (\ref{dual})-(\ref{dual2}) has
the $H^4$ regularity. Then, there exists a constant $C$ such that
\begin{equation}\label{err2}
\|Q_0u-u_0\| \le Ch^{k+t_0-2}(\|u\|_{k+1}+ h\delta_{k,2}\|u\|_4).
\end{equation}
In other words, for quadratic elements (i.e., $k=2$) we have a
sub-optimal order of convergence given by $\|Q_0u-u_0\| =
\mathcal{O}(h^2)$. But for cubic or higher order of elements, we
have the optimal order of convergence $\|Q_0u-u_0\| =
\mathcal{O}(h^{k+1})$.
\end{theorem}
\begin{proof}
Testing (\ref{dual}) with the error function $e_0$ on each element
and then using integration by parts to obtain
\begin{eqnarray*}
\|e_0\|^2&=&(\Delta^2 w, e_0)\\
&=&\sum_{T\in\T_h}(\Delta w,\Delta e_0)_T+\sum_{T\in\T_h}\l \nabla (\Delta w)\cdot\bn, e_0\r_\pT-\sum_{T\in\T_h}\l\Delta w,\nabla e_0\cdot\bn\r_\pT\\
&=&\sum_{T\in\T_h}(\Delta w,\Delta e_0)_T+\sum_{T\in\T_h}\l \nabla
(\Delta w)\cdot\bn, e_0-e_b\r_\pT-\sum_{T\in\T_h}\l\Delta w,(\nabla
e_0-\be_{g})\cdot\bn\r_\pT,
\end{eqnarray*}
where we have used the fact that $e_b$ and $\be_g\cdot\bn$ vanishes
on the boundary of $\Omega$. Using (\ref{key}) with $w$ in the place
of $u$, we can rewrite the above identity as follows
\begin{eqnarray*}
\|e_0\|^2&=&(\Delta_w Q_hw,\Delta_w e_h)_h+\sum_{T\in\T_h}\l (\nabla
(\Delta
w)-\nabla (\bbQ_h\Delta w))\cdot\bn, e_0-e_b\r_\pT\\
& &-\sum_{T\in\T_h}\l\Delta w-\bbQ_h\Delta w,(\nabla
e_0-\be_{g})\cdot\bn\r_\pT.
\end{eqnarray*}
Next, we have from the error equation (\ref{ee}) that
\begin{eqnarray*}
( \Delta_w Q_hw, \Delta_{w} e_h)_h&=&\sum_{T\in\T_h}\langle \Delta
u-\bbQ_h\Delta u,
(\nabla Q_0w-\bQ_{g}\nabla w)\cdot\bn\rangle_{\pT}\\
&-& \sum_{T\in\T_h}\langle\nabla(\Delta u-\bbQ_h\Delta u)\cdot\bn,
Q_0w-Q_bw\rangle_{\pT}\\
&-&s(e_h,Q_hw)+s(Q_hu,Q_hw).
\end{eqnarray*}
Combining the above two equations we obtain
\begin{eqnarray}
\|e_0\|^2&=&\sum_{T\in\T_h}\l (\nabla (\Delta w)-\nabla (\bbQ_h\Delta w))\cdot\bn,
e_0-e_b\r_\pT\label{lll}\\
&-&\sum_{T\in\T_h}(\Delta w-\bbQ_h\Delta w,(\nabla e_0-\be_{g})\cdot\bn\r_\pT-s(e_h,Q_hw)\nonumber\\
&+&\sum_{T\in\T_h}\langle\Delta u-\bbQ_h\Delta u,
(\nabla Q_0w-\bQ_{g}\nabla w)\cdot\bn\rangle_{\pT}\nonumber\\
&-& \sum_{T\in\T_h}\langle\nabla(\Delta u-\bbQ_h\Delta u)\cdot\bn,
Q_0w-Q_bw\rangle_{\pT}+s(Q_hu,Q_hw).\nonumber
\end{eqnarray}
Each of the six terms on the right-hand side of (\ref{lll}) can be
bounded by using the Cauchy-Schwarz inequality and Lemma \ref{l2} as
follows.

For the first term, it follows from the estimate (\ref{mmm2}) and
the fact $t_0=\min(3,k)\le 3$ that
\begin{eqnarray}
&&\left|\sum_{T\in\T_h}\langle \nabla(\Delta w-\bbQ_h\Delta
w)\cdot\bn, e_0-e_b\rangle_\pT\right| \label{March02:102}\\
\le && \left(\sum_{T\in\T_h}h_T^3\|\nabla(\Delta w-\bbQ_h\Delta
w)\|_\pT^2\right)^{\frac12}
\left(\sum_{T\in\T_h}h_T^{-3}\|e_0-e_b\|^2_{\pT}\right)^{\frac12}\nonumber\\
\le && C h^{t_0-1}(\|w\|_{t_0+1} + h\delta_{t_0,2} \|w\|_4) \3bar
e_h\3bar\nonumber\\
\le && C h^{t_0-1}\|w\|_4 \3bar e_h\3bar.\nonumber
\end{eqnarray}

As to the second term, we use the Cauchy-Schwarz inequality and the
estimate (\ref{mmm1}) with $m=t_0$ to obtain
\begin{eqnarray}
&&\left|\sum_{T\in\T_h} \langle \Delta w-\bbQ_h\Delta w, (\nabla
e_0-\be_{g})\cdot\bn\rangle_\pT\right|\label{March02:101}\\
\le &&\left(\sum_{T\in\T_h} h_T\|\Delta w-\bbQ_h\Delta
w\|_\pT^2\right)^{\frac12} \left(\sum_{T\in\T_h} h_T^{-1}
\|\nabla e_0-\be_{g}\|_\pT^2\right)^{\frac12}\nonumber\\
\le && C h^{t_0-1}\|w\|_{t_0+1} \3bar e_h\3bar\nonumber\\
\le && C h^{t_0-1}\|w\|_{4} \3bar e_h\3bar.\nonumber
\end{eqnarray}
Analogously, for the third term, we have
\begin{equation}\label{March02:181}
|s(e_h,Q_hw)|\le Ch^{t_0-1}\|w\|_{4}\3bar e_h\3bar.
\end{equation}

The fourth term can be bounded by using Lemma \ref{l1} as follows
\begin{eqnarray}\label{March02:158}
&&\left|\sum_{T\in\T_h}\langle\Delta u-\bbQ_h\Delta u,
(\nabla Q_0w-\bQ_{g}\nabla w)\cdot\bn\rangle_{\pT}\right|\\
\le &&\left(\sum_{T\in\T_h}h\|\Delta u-\bbQ_h\Delta
u\|_\pT^2\right)^{\frac12}
\left(\sum_{T\in\T_h}h^{-1}\|\nabla Q_0w-\nabla w\|^2_{\pT}\right)^{\frac12}\nonumber\\
\le &&Ch^{k-1}\|u\|_{k+1} h^{t_0-1}\|w\|_{t_0+1}\nonumber\\
\le &&Ch^{k+t_0-2}\|u\|_{k+1} \|w\|_{4}.\nonumber
\end{eqnarray}

As to the fifth term, we again use the Cauchy-Schwarz inequality and
Lemma \ref{l2} to obtain
\begin{eqnarray}\label{March02:159}
&&\left|\sum_{T\in\T_h}\langle\nabla(\Delta u-\bbQ_h\Delta
u)\cdot\bn, Q_0w-Q_bw\rangle_{\pT}\right|\\
&&\le \left(\sum_{T\in\T_h}h_T^3\|\nabla(\Delta u-\bbQ_h\Delta
u)\|_\pT^2
\right)^{\frac12}\left(\sum_{T\in\T_h}h_T^{-3}\|Q_0w-w \|^2_{\pT}\right)^{\frac12}\nonumber\\
&&\le Ch^{k-1}(\|u\|_{k+1}+h\delta_{k,2}\|u\|_4)
h^{t_0-1}\|w\|_{t_0+1}\nonumber\\
&&\le Ch^{k+t_0-2}(\|u\|_{k+1}+h\delta_{k,2}\|u\|_4)
\|w\|_{4}.\nonumber
\end{eqnarray}

The last term on the right-hand side of (\ref{lll}) can be estimated
as follows.
\begin{eqnarray}\label{March02:161}
&&|s(Q_hu, Q_hw)|\\
&\le&\left|\sum_{T\in\T_h} h_T^{-1}\langle \nabla
Q_0u-\bQ_{g}\nabla u, \nabla Q_0w -\bQ_{g} \nabla w\rangle_\pT\right|\nonumber \\
&+& \left|\sum_{T\in\T_h} h_T^{-3}\langle Q_0u-Q_{b}u, Q_0w-Q_{b}w\rangle_\pT\right|\nonumber\\
&\le&\left(\sum_{T\in\T_h}h_T^{-1}\|\nabla Q_0u-\nabla
u\|_\pT^2\right)^{\frac12}
\left(\sum_{T\in\T_h}h_T^{-1}\|\nabla Q_0w-\nabla w\|^2_{\pT}\right)^{\frac12}\nonumber\\
&+&\left(\sum_{T\in\T_h}h_T^{-3}\|Q_0u-u\|_\pT^2\right)^{\frac12}
\left(\sum_{T\in\T_h}h_T^{-3}\|Q_0w-w\|^2_{\pT}\right)^{\frac12}\nonumber\\
&\le&Ch^{k-1}\|u\|_{k+1} \ h^{t_0-1}\|w\|_{t_0+1}\nonumber\\
&\le&Ch^{k+t_0-2}\|u\|_{k+1} \|w\|_{4}.\nonumber
\end{eqnarray}
Substituting the estimates (\ref{March02:102})-(\ref{March02:161})
into (\ref{lll}) yields
$$
\|e_0\|^2\le Ch^{t_0-1} \3bar e_h\3bar \|w\|_4 +
Ch^{k+t_0-2}(\|u\|_{k+1}+ h\delta_{k,2}\|u\|_4)\|w\|_4.
$$
Using the regularity estimate (\ref{reg}) we arrive at
$$
\|e_0\|\le Ch^{t_0-1} \3bar e_h\3bar + Ch^{k+t_0-2}(\|u\|_{k+1}+
h\delta_{k,2}\|u\|_4).
$$
Finally, by combining the above estimate with the $H^2$ error
estimate (\ref{err1}) we obtain the desired $L^2$ error estimate
(\ref{err2}).
\end{proof}

\section{Numerical Results}\label{Section:numerical-results}

Our numerical experiments are conducted for the weak Galerkin finite
element scheme (\ref{wg-2}) by using the following finite element
space
$$
\tilde V_h=\{v=\{v_0,v_b,v_{g}\bn_e\}, \ v_0\in P_2(T),\ v_b\in
P_2(e),\ v_{g}\in P_1(e), T\in\mathcal{T}_h, e\in\mathcal{E}_h\}.
$$
For any given $v=\{v_0,v_b,v_{g}\bn_e\}\in \tilde V_h,$ the discrete
weak Laplacian, $\Delta_w v$, is computed locally on each element
$T$ as a function in $P_0(T)$ by solving the following equation
$$
(\Delta_w v,\psi)_T=(v_0,\Delta\psi)_T+\langle v_{g}{\bf
n}_e\cdot{\bf n}, \psi\rangle_{\partial T}-\langle
v_b,\nabla\psi\cdot{\bf n}\rangle_{\partial T},
$$
for all $\psi\in P_0(T)$. Since $\psi\in P_0(T)$, the above equation
can be simplified as
\begin{equation}
(\Delta_w v,\psi)_T=\langle v_{g}{\bf n}_e\cdot{\bf
n},\psi\rangle_{\partial T}.
\end{equation}

The error for the WG-FEM solution is measured in two norms defined
as follows:
\begin{eqnarray}
\3bar v_h \3bar^2:&=&\sum_{T\in\mathcal{T}_h}\bigg(\int_T|\Delta_w v_h|^2dx+
h_T^{-1}\int_{\partial T}|(\nabla v_0)\cdot\bn_e-v_{g}|^2ds \\
&&+h_T^{-3}\int_{\partial T}(v_0-v_b)^2ds\bigg),\qquad\qquad\  (\mbox{A discrete $H^2$-norm}),\notag\\
\|v_h\|^2:&=&\sum_{T\in\mathcal{T}_h}\int_T|v_0|^2dx,\qquad\qquad\qquad\qquad  (\mbox{Element-based $L^2$-norm}).
\end{eqnarray}

\subsection{Test Case 1}

Here we consider the forth order problem that seeks an unknown
function $u=u(x,y)$ satisfying
$$
\Delta^2 u=f
$$
in the square domain $\Omega=(0,1)^2$ with homogeneous Dirichlet and
Neumann boundary conditions. The exact solution is given by
$u=x^2(1-x)^2y^2(1-y)^2$, and the function $f=f(x,y)$ is computed to
match the exact solution. Uniform triangular meshes are used and
they are constructed as follows: (1) partition the domain into
$n\times n$ sub-rectangles; (2) divide each square element into two
triangles by the diagonal line with a negative slope. The mesh size
is denoted by $h=1/n.$ Table \ref{ex1_tri} shows that the
convergence rates for the WG-FEM solution in the $H^2$ and $L^2$
norms are of order $O(h)$ and $O(h^2)$, respectively.

\begin{table}[h]
\caption{Test Case 1: Numerical error and convergence rates in $H^2$
and $L^2$.}\label{ex1_tri} \center
\begin{tabular}{|c||c|c|c|c|}
\hline
$h$ & $\3bar u_h-Q_h u\3bar$ & order & $\|u_0-Q_0 u\|$ & order \\
\hline\hline
   2.5000e-01  &2.5683e-01 &           &3.3304e-02 &       \\ \hline
   1.2500e-01  &1.3540e-01 &9.2359e-01 &9.1046e-03 &1.8710 \\ \hline
   6.2500e-02  &7.2378e-02 &9.0360e-01 &2.6049e-03 &1.8054 \\ \hline
   3.1250e-02  &3.8275e-02 &9.1915e-01 &7.3257e-04 &1.8302 \\ \hline
   1.5625e-02  &1.9687e-02 &9.5916e-01 &1.9461e-04 &1.9124 \\ \hline
   7.8125e-03  &9.9457e-03 &9.8510e-01 &4.9762e-05 &1.9675 \\ \hline
\end{tabular}
\end{table}

\subsection{Test Case 2}
Here we solve the biharmonic equation on the domain of unit square
with nonhomogeneous Dirichlet and Neumann boundary conditions. The
exact solution is given by $u=\sin(\pi x)\sin(\pi y)$, and the
function $f=f(x,y)$ is computed accordingly.

The mesh triangulations are constructed in the same way as in the
test case 1. The numerical results are presented in Table
\ref{ex2_tri} which confirm the theory developed in earlier
sections.
\begin{table}[h]
\caption{Test Case 2: Numerical error and convergence rates for the
biharmonic equation with non-homogeneous boundary
conditions.}\label{ex2_tri} \center
\begin{tabular}{|c||c|c|c|c|}
\hline
$h$ & $\3bar u_h-Q_h u\3bar$ & order & $\|u_0-Q_0 u\|$ & order \\
\hline\hline
   2.5000e-01   &2.4536e+01 &            &3.1862e+00 &        \\ \hline
   1.2500e-01   &1.2794e+01 &9.3943e-01  &8.5298e-01 &1.9013  \\ \hline
   6.2500e-02   &6.7243e+00 &9.2801e-01  &2.3439e-01 &1.8636  \\ \hline
   3.1250e-02   &3.4811e+00 &9.4984e-01  &6.2578e-02 &1.9052  \\ \hline
   1.5625e-02   &1.7657e+00 &9.7930e-01  &1.6066e-02 &1.9616  \\ \hline
   7.8125e-03   &8.8709e-01 &9.9309e-01  &4.0534e-03 &1.9868  \\ \hline
\end{tabular}
\end{table}

More numerical experiments should be conducted for the WG-FEM
schemes (\ref{wg}) and (\ref{wg-2}). In particular, it would be
interesting to see results for WG approximations of high order
elements. It is also important to hybridize (\ref{wg}) and
(\ref{wg-2}) so that unknowns related to $v_0$ can be eliminated
locally on each element.

\appendix

\section*{Appendix}\label{Section:appendix}
\setcounter{section}{1} The goal of this Appendix is to establish
some fundamental estimates useful in the error estimate for general
weak Galerkin finite element methods. First, we derive a trace
inequality for functions defined on the finite element partition
$\T_h$ with properties as specified in \cite{wy-mixed}. For
completeness, we review the shape regularity assumption in the
following subsection.

\subsection{Domain Partition and Shape Regularity}
Let ${\cal T}_h$ be a partition of the domain $\Omega$ consisting of
polygons in two dimensions or polyhedra in three dimensions
satisfying a set of conditions to be specified. Denote by ${\cal
E}_h$ the set of all edges or flat faces in ${\cal T}_h$, and let
${\cal E}_h^0={\cal E}_h\backslash\partial\Omega$ be the set of all
interior edges or flat faces. For every element $T\in \T_h$, we
denote by $|T|$ the area or volume of $T$ and by $h_T$ its diameter.
Similarly, we denote by $|e|$ the length or area of $e$ and by $h_e$
the diameter of edge or flat face $e\in\E_h$. We also set as usual
the mesh size of $\T_h$ by
$$
h=\max_{T\in\T_h} h_T.
$$

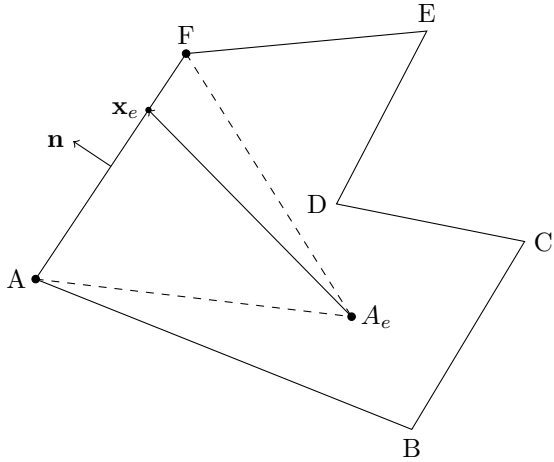
\begin{figure}[h!]
\begin{center}
\begin{tikzpicture}
\coordinate (A) at (-3,0); \coordinate (B) at (2,-2); \coordinate
(C) at (3.5, 0.5); \coordinate (D) at (1,1); \coordinate (E) at
(2.2, 3.3); \coordinate (F) at (-1, 3); \coordinate (CC) at (0,0);
\coordinate (Ae) at (1.2,-0.5); \coordinate (xe) at (-1.5, 2.25);
\coordinate (AFc) at (-2, 1.5); \coordinate (AFcLeft) at (-2.5,
1.83333); \draw node[left] at (xe) {$\bx_e$}; \draw node[right] at
(Ae) {$A_e$}; \draw node[left] at (A) {A}; \draw node[below] at (B)
{B}; \draw node[right] at (C) {C}; \draw node[left] at (D) {D};
\draw node[above] at (E) {E}; \draw node[above] at (F) {F}; \draw
node[left] at (AFcLeft) {$\bn$}; \draw
(A)--(B)--(C)--(D)--(E)--(F)--cycle; \draw[dashed](A)--(Ae);
\draw[dashed](F)--(Ae); \draw[->] (Ae)--(xe); \draw[->]
(AFc)--(AFcLeft); \filldraw[black] (A) circle(0.05);
    \filldraw[black] (F) circle(0.05);
    \filldraw[black] (Ae) circle(0.05);
    \filldraw[black] (xe) circle(0.035);
\end{tikzpicture}
\caption{Depiction of a shape-regular polygonal element $ABCDEFA$.}
\label{fig:shape-regular-element}
\end{center}
\end{figure}

All the elements of $\T_h$ are assumed to be closed and simply
connected polygons or polyhedra, see Fig.
\ref{fig:shape-regular-element}. The partition $\T_h$ is said to be
shape regular if the following properties are satisfied.

\medskip
\begin{description}
\item[A1:] \ Assume that there exist two positive constants $\varrho_v$ and $\varrho_e$
such that for every element $T\in\T_h$ we have
\begin{equation}\label{a1}
\varrho_v h_T^d\leq |T|,\qquad \varrho_e h_e^{d-1}\leq |e|
\end{equation}
for all edges or flat faces of $T$.

\item[A2:] \ Assume that there exists a positive constant $\kappa$ such that for every element
$T\in\T_h$ we have
\begin{equation}\label{a2}
\kappa h_T\leq h_e
\end{equation}
for all edges or flat faces $e$ of $T$.

\item[A3:] \ Assume that the mesh edges or faces are flat. We
further assume that for every $T\in\T_h$, and for every edge/face
$e\in \partial T$, there exists a pyramid $P(e,T, A_e)$ contained in
$T$ such that its base is identical with $e$, its apex is $A_e\in
T$, and its height is proportional to $h_T$ with a proportionality
constant $\sigma_e$ bounded from below by a fixed positive number
$\sigma^*$. In other words, the height of the pyramid is given by
$\sigma_e h_T$ such that $\sigma_e\ge \sigma^*>0$. The pyramid is
also assumed to stand up above the base $e$ in the sense that the
angle between the vector $\bx_e-A_e$, for any $x_e\in e$, and the
outward normal direction of $e$ is strictly acute by falling into an
interval $[0, \theta_0]$ with $\theta_0< \frac{\pi}{2}$.

\item[A4:] \ Assume that each $T\in\T_h$ has a circumscribed simplex $S(T)$ that is
shape regular and has a diameter $h_{S(T)}$ proportional to the
diameter of $T$; i.e., $h_{S(T)}\leq \gamma_* h_T$ with a constant
$\gamma_*$ independent of $T$. Furthermore, assume that each
circumscribed simplex $S(T)$ intersects with only a fixed and small
number of such simplices for all other elements $T\in\T_h$.
\end{description}

\subsection{A Trace Inequality}

The following is trace inequality concerning functions in the
Sobolev space $W^{1,p}$.

\begin{lemma}[Trace Inequality]\label{Lemma:trace-inequality}
Let $\T_h$ be a partition of the domain $\Omega$ into
polygons in 2D or polyhedra in 3D. Assume that the partition $\T_h$
satisfies the assumptions (A1), (A2), and (A3) as specified in
Section \cite{wy-mixed}. Let $p>1$ be any real number. Then, there
exists a constant $C$ such that for any $T\in\T_h$ and edge/face
$e\in\partial T$, we have
\begin{equation}\label{App:trace}
\|\theta\|_{L^p(e)}^p \leq C h_T^{-1}\left(\|\theta\|_{L^p(T)}^p +
h_T^{p} \|\nabla \theta\|_{L^p(T)}^p\right),
\end{equation}
where $\theta\in W^{1,p}(T)$ is any function.
\end{lemma}
\begin{proof}
We shall provide a proof for the case of $e$ being a flat face and
$\theta\in C^1(T)$. To this end, let the flat face $e$ be
represented by the following parametric equation:
\begin{eqnarray}\label{parametric-e}
\bx_e=\phi(\xi,\eta):=(\phi_1(\xi, \eta),
\phi_2(\xi,\eta),\phi_3(\xi,\eta))
\end{eqnarray}
for $(\xi,\eta)\in D\subset \mathbb{R}^2$. By Assumption {\bf A3},
there exists a pyramid $P(e,T,A_e)$ with apex
$A_e=\bx_*:=(x_1^*,x_2^*,x_3^*)$ contained in the element $T$. This
pyramid has the following parametric representation:
\begin{eqnarray}\label{pyramid}
\bx(t,\xi,\eta)&=&(1-t)\phi(\xi, \eta)+t\bx_*
\end{eqnarray}
for $(t, \xi,\eta)\in [0,1]\times D$. For any given $\bx_e\in e$,
the line segment joining $\bx_e$ and the apex $\bx_*$ can be
represented by
$$
\bx(t) = \bx_e + t(\bx_*-\bx_e).
$$
From the fundamental theorem of Calculus,  we have
$$
|\theta|^p(\bx_e)-|\theta|^p(\bx(t))=-\int_0^{t}\partial_\tau
(|\theta|^p(\bx_e+\tau\omega))d\tau, \qquad\omega = {\bx_*-\bx_e}.
$$
The above can be further rewritten as
$$
|\theta|^p(\bx_e)-|\theta|^p(\bx(t))=-p\int_0^{t}|\theta|^{p-1}\sgn(\theta) \left(
\nabla\theta\cdot \omega\right) d\tau.
$$
Let $q=\frac{p}{p-1}$. It follows from the Cauchy-Schwarz inequality that for any $t\in
[0,\frac12]$ we have
$$
|\theta|^p(\bx_e) \le |\theta|^p(\bx(t)) + p
\left(\int_0^{\frac12}|\theta|^p |\omega| d\tau \right)^{\frac{1}{q}} \left(\int_0^{\frac12}
|\nabla\theta|^{p}|\omega| \right)^{\frac{1}{p}}.
$$
Using the Young's inequality $ab\le \varepsilon a^p+\varepsilon^{-\frac{q}{p}}{b^q}$ for any non-negative real numbers $a, \ b$, and $\varepsilon>0$, we arrive at
$$
|\theta|^p(\bx_e) \le |\theta|^p(\bx(t)) + p\ \varepsilon^{-\frac{q}{p}}
\int_0^{\frac12}|\theta|^p |\omega| d\tau + p\ \varepsilon \int_0^{\frac12}
|\nabla\theta|^{p}|\omega|d\tau.
$$
Now we integrate the above inequality over the flat face $e$ by
using the parametric equation (\ref{parametric-e}), yielding
\begin{eqnarray}
\int_D|\theta|^p(\bx_e)|\phi_\xi\times\phi_\eta| d\xi d\eta &\le&
\int_D|\theta|^p(\bx(t))|\phi_\xi\times\phi_\eta| d\xi d\eta \nonumber\\
& & +p\ \varepsilon^{-\frac{q}{p}} \int_0^{\frac12}\int_D|\theta|^p |\omega|
|\phi_\xi\times\phi_\eta|
d\xi d\eta d\tau \label{aaa.01}\\
& & +p\ \varepsilon \int_0^{\frac12}\int_D|\nabla\theta|^p |\omega|
|\phi_\xi\times\phi_\eta| d\xi d\eta d\tau.\label{aaa.02}
\end{eqnarray}
Observe that the integral of a function over the following
prismatoid
$$
P_{\frac12}:=\left\{\bx(t,\xi,\eta):\quad (t,\xi,\eta)\in
[0,1/2]\times D \right\}
$$
is given by
$$
\int_{P_{\frac12}} f(\bx) d\bx= \int_{0}^{\frac12}\int_D
f(\bx(\tau,\xi,\eta)) J(\tau,\xi,\eta) d\xi d\eta d\tau,
$$
where $J(\tau,\xi,\eta)=(1-\tau)^2 |(\phi_\xi\times \phi_\eta)\cdot
\omega|$ is the Jacobian from the coordinate change. The vector
$\phi_\xi\times \phi_\eta$ is normal to the face $e$, and
$\omega=\bx_*-\bx_e$ is a vector from the base point $\bx_e$ to the
apex $\bx_*$. The angle assumption (see Assumption {\bf A3} of
Section \ref{wg-mfem}) for the prism $P(e,T,A_e)$
indicates that the Jacobian satisfies the following relation
\begin{equation}\label{jacobian-estimate}
J(\tau,\xi,\eta)\ge \frac{\mu_0}{4} |\phi_\xi\times \phi_\eta|\
|\omega|,\quad \tau\in [0,1/2]
\end{equation}
for some fixed $\mu_0\in (0,1)$. Observe that the
left-hand side of (\ref{aaa.01}) is the surface integral of
$|\theta|^p$ over the face $e$. Thus, substituting
(\ref{jacobian-estimate}) into (\ref{aaa.01}) and (\ref{aaa.02})
yields
\begin{eqnarray*}
\int_e|\theta|^p de &\le&
\int_D|\theta|^p(\bx(t))|\phi_\xi\times\phi_\eta| d\xi d\eta \nonumber\\
& & +4p\mu_0^{-1}\varepsilon^{-\frac{q}{p}}\int_{P_{\frac12}} |\theta|^p d\bx
+4p\mu_0^{-1}\varepsilon \int_{P_{\frac12}}|\nabla\theta|^p d\bx.
\end{eqnarray*}
Now we integrate the above with respect to $t$ in the interval
$[0,\frac12]$ to obtain
\begin{eqnarray}
\frac12 \int_e|\theta|^p de &\le&
\int_0^{\frac12}\int_D|\theta|^p(\bx(t))|\phi_\xi\times\phi_\eta| d\xi d\eta dt\label{aaa.03}\\
& & +2p\mu_0^{-1}\varepsilon^{-\frac{q}{p}}\int_{P_{\frac12}} |\theta|^p d\bx
+2p\mu_0^{-1}\varepsilon\int_{P_{\frac12}}|\nabla\theta|^p d\bx.\nonumber
\end{eqnarray}
Again, by substituting (\ref{jacobian-estimate}) into the right-hand
side of (\ref{aaa.03}) we arrive at
\begin{eqnarray}
\frac12 \int_e|\theta|^p de &\le& 4\mu_0^{-1} |\omega|^{-1}
\int_0^{\frac12}\int_D|\theta|^p(\bx(t))J(t,\xi,\eta) d\xi d\eta dt
\label{aaa.04}\\
& & +2p\mu_0^{-1}\varepsilon^{-\frac{q}{p}}\int_{P_{\frac12}} |\theta|^p d\bx
+2p\mu_0^{-1}\varepsilon \int_{P_{\frac12}}|\nabla\theta|^p d\bx.\nonumber
\end{eqnarray}
The first integral on the right-hand side of (\ref{aaa.04}) is the
integral of $|\theta|^p$ on the prismatoid $P_{\frac12}$. It can be
seen from the Assumption {\bf A3} that
\begin{equation}\label{yes.888}
|\omega|^{-1} \leq \alpha_* h_T^{-1}
\end{equation}
for some positive constant $\alpha_*$. By taking $\varepsilon= h_T^{p-1}$, we have
$\varepsilon^{-\frac{q}{p}}=h_T^{-1}$. It then follows from (\ref{aaa.04}) and (\ref{yes.888}) that
$$
\int_e |\theta|^p de \leq C h_T^{-1} \left(\int_{P_{\frac12}}
|\theta|^p d\bx + h_T^{p}\int_{P_{\frac12}} |\nabla\theta|^p
d\bx\right),
$$
which completes the proof of the Lemma.
\end{proof}
\medskip

\subsection{A Domain Inverse Inequality}

Next, we would like to establish an estimate for the $L^p$ norm of
polynomial functions by their $L^p$ norm on a subdomain. To this
end, we first derive a result of similar nature for general
functions in $W^{1,p}$.

\begin{lemma}\label{appendix-lemma2}
Let $K\subset\mathbb{R}^d$ be convex and $v\in W^{1,p}(K)$ with
$p\ge 1$. Then,
\begin{equation}\label{aaa.08}
\|v\|^p_{L^p(K)} \leq \frac{2|K|}{|S|} \|v\|_{L^p(S)}^p
+\left(\frac{2p\omega_d\delta^{d+1}}{|S|} \right)^p\|\nabla
v\|_{L^p(K)}^p,
\end{equation}
where $\delta$ is the diameter of $K$, $S$ is any measurable subset
of $K$, and $\omega_d=\frac{2\pi^{d/2}}{d\Gamma(d/2)}$ is the volume
of unit ball in $\mathbb{R}^d$.
\end{lemma}

\begin{proof}
Since $C^1(K)$ is dense in $W^{1,p}(K)$, it is sufficient to establish
(\ref{aaa.08}) for $v\in C^1(K)$. For any $\bx, \by\in K$, we have
$$
|v|^p(\bx)=|v|^p(\by) - \int_0^{|\bx-\by|} \partial_r
(|v|^p(\bx+r\omega) )dr,\quad \omega=\frac{\by-\bx}{|\by-\bx|}.
$$
From the usual chain rule and the Cauchy-Schwarz inequality we obtain
\begin{eqnarray}
|v|^p(\bx)&=&|v|^p(\by) - p\int_0^{|\bx-\by|} |v|^{p-1} \sgn(v) \partial_r v(\bx+r\omega)
dr\nonumber\\
&\le & |v|^p(\by) + p\ \varepsilon^{-\frac{q}{p}} \int_0^{|\bx-\by|} |v|^p dr + p\ \varepsilon\int_0^{|\bx-\by|} |\partial_r v|^p dr,\label{aaa.10}
\end{eqnarray}
where $\varepsilon>0$ is any constant. Let
$$
V(\bx)=\left\{
\begin{array}{ll}
|v|^p(\bx),&\qquad\bx\in K\\
0,&\qquad\bx\notin K
\end{array}
\right.
$$
and
$$
W(\bx)=\left\{
\begin{array}{ll}
|\partial_r v|^p(\bx),&\qquad\bx\in K\\
0,&\qquad\bx\notin K.
\end{array}
\right.
$$
Then, the inequality (\ref{aaa.10}) can be rewritten as
$$
|v|^p(\bx)\le |v|^p(\by) + p\ \varepsilon^{-\frac{q}{p}} \int_0^{\infty} V(\bx+r\omega)dr + p\ \varepsilon\int_0^{\infty}
W(\bx+r\omega)dr.
$$
Integrating the above inequality with respect to $\by$ in $S$ yields
\begin{eqnarray}
& & |S| |v|^p(\bx) \leq \int_S |v|^p dS \label{aaa.16}\\
& & \ + p \int_{|\bx-\by|\le \delta}\left(\varepsilon^{-\frac{q}{p}}\int_0^{\infty} V(\bx+r\omega)dr+\varepsilon \int_0^{\infty} W(\bx+r\omega)dr  \right)d\by.\nonumber
\end{eqnarray}
It is not hard to see that
\begin{eqnarray}
\int_{|\bx-\by|\le \delta}\int_0^{\infty} V(\bx+r\omega)dr d\by &=& \int_0^\infty\int_{|\omega|=1}\int_0^\delta V(\bx+r\omega) \rho^{d-1}\ d\rho\ d\omega\ dr\nonumber\\
&=&\frac{\delta^d}{d}\int_0^\infty\int_{|\omega|=1}V(\bx+r\omega) d\omega\ dr\nonumber\\
&=&\frac{\delta^d}{d}\int_K |\bx-\by|^{1-d}
|v|^p(\by)d\by.\label{aaa.17}
\end{eqnarray}
Analogously, we have
\begin{equation}\label{aaa.18}
\int_{|\bx-\by|\le \delta}\int_0^{\infty} W(\bx+r\omega)dr d\by
=\frac{\delta^d}{d}\int_K |\bx-\by|^{1-d} |\partial_r v(\by)|^pd\by.
\end{equation}
Substituting (\ref{aaa.17}) and (\ref{aaa.18}) into (\ref{aaa.16}) yields
\begin{eqnarray*}
& & |S|\ |v|^p(\bx) \\
\leq && \int_S |v|^p dS + \frac{p\delta^d}{d}\left(
\varepsilon^{-\frac{q}{p}}\int_K |\bx-\by|^{1-d} |v|^p(\by)d\by
+\varepsilon\int_K |\bx-\by|^{1-d} |\nabla v(\by)|^pd\by\right).
\end{eqnarray*}
Observe that the following holds true
$$
\int_K|\bx-\by|^{1-d}d\bx \leq \delta S_{d-1},
$$
where $S_{d-1}$ is the ``area" of the unit surface in
$\mathbb{R}^d$. The volume of the unit ball in $\mathbb{R}^d$,
$\omega_d$, is related to $S_{d-1}$ as follows
$$
\omega_d=\frac{S_{d-1}}{d}.
$$
Now integrating both sides with respect to $\bx$ in $K$ gives
$$
|S|\int_K |v|^pdK\leq |K|\int_S |v|^p dS
+p\omega_d\delta^{d+1}\left(\varepsilon^{-\frac{q}{p}}\int_K |v|^pdK
+\varepsilon\int_K |\nabla v|^pdK\right),
$$
which yields the desired estimate (\ref{aaa.08}) by setting
$\varepsilon=
\left(\frac{2p\omega_d\delta^{d+1}}{|S|}\right)^{p-1}$.
\end{proof}
\medskip

Consider a case of Lemma \ref{appendix-lemma2} in which the convex
domain $K$ is a shape regular $d$-simplex. Denote by $h_K$ the
diameter of $K$. The shape regularity implies that
\begin{enumerate}
\item the measure of $K$ is proportional to $h_K^d$,
\item there exists an inscribed ball $B_K\subset K$ with diameter proportional to $h_K$.
\end{enumerate}
Now let $S$ be a ball inside of $K$ with radius $r_S\ge \varsigma_*
h_K$. Then, there exists a fixed constant $\kappa_*$ such that
\begin{equation}\label{aaa.21}
|K|\leq \kappa_* |S|.
\end{equation}
Apply (\ref{aaa.21}) in (\ref{aaa.08}) and notice that $\tau_* |S|\ge h_K^d$ and $\delta = h_K$. Thus,
\begin{equation}\label{aaa.28}
\|v\|^p_{L^p(K)} \leq 2 \kappa_* \|v\|_{L^p(S)}^p +\left(2p\tau_*
h_K\omega_d\right)^p\|\nabla v\|_{L^p(K)}^p.
\end{equation}
For simplicity of notation, we shall rewrite (\ref{aaa.28}) in the following form
\begin{equation}\label{aaa.38}
\|v\|^p_{K,p} \leq a_0 \|v\|_{S,p}^p+a_1 h_K^{p}\|\nabla v\|_{K,p}^p.
\end{equation}
If $v$ is infinitely smooth, then a recursive use of the estimate (\ref{aaa.38}) yields the following result
\begin{equation}\label{aaa.39}
\|v\|^p_{K,p} \leq \sum_{j=0}^n a_j h_K^{jp}\|\nabla^j v\|_{S,p}^p+a_{n+1} h_K^{pn+p}\|\nabla^{n+1} v\|_{K,p}^p.
\end{equation}
In particular, if $v$ is a polynomial of degree $n$, then
\begin{equation}\label{aaa.48}
\|v\|^p_{K,p} \leq \sum_{j=0}^n a_j h_K^{jp}\|\nabla^j v\|_{S,p}^p.
\end{equation}
The standard inverse inequality implies that
$$
\|\nabla^j v\|_{S,p} \lesssim h_K^{-j}\|v\|_{S,p}.
$$
Substituting the above into (\ref{aaa.48}) gives
\begin{equation}\label{aaa.58}
\|v\|^p_{K,p} \lesssim \|v\|_{S,p}^p.
\end{equation}
The result is summarized as follows.
\begin{lemma}[Domain Inverse Inequality]\label{appendix.thm}
Let $K\subset\mathbb{R}^d$ be a $d$-simplex which has diameter $h_K$ and is shape regular.
Assume that $S$ is a ball in $K$ with diameter $r_S$ proportional to $h_K$; i.e., $r_S\ge \varsigma_* h_K$ with a fixed $\varsigma_*>0$. Then, there exists a constant $C=C(\varsigma_*, n)$ such that
\begin{equation}\label{aaa.58-new}
\|v\|_{L^p(K)} \le C(\varsigma_*,n) \|v\|_{L^p(S)}
\end{equation}
for any polynomial $v$ of degree no more than $n$.
\end{lemma}

\subsection{Inverse Inequalities}
The usual inverse inequality in finite element analysis also holds
true for piecewise polynomials defined on the finite element
partition $\T_h$ provided that it satisfies the assumptions {\bf
A1-A4}.

\begin{lemma}\label{appendix.inverse-inq}
Let $\T_h$ be a finite element partition of $\Omega$ consisting of
polygons or polyhedra. Assume that $\T_h$ satisfies all the
assumptions {\bf A1-A4} and $p\ge 1$ be any real number. Then, there
exists a constant $C=C(n)$ such that
\begin{equation}\label{aaa.88-new}
\|\nabla \varphi\|_{T,p} \le C(n) h^{-1}_T \|\varphi\|_{T,p},\qquad
\forall T\in \T_h
\end{equation}
for any piecewise polynomial $\varphi$ of degree $n$ on $\T_h$.
\end{lemma}
\begin{proof}
The proof is merely a combination of Lemma \ref{appendix.thm} and
the standard inverse inequality on d-simplices. To this end, for any
$T\in \T_h$, let $S(T)$ be the circumscribed simplex that is shape
regular. It follows from the standard inverse inequality that
$$
\|\nabla \varphi\|_{T,p}\leq \|\nabla\varphi\|_{S(T),p}\leq
Ch_T^{-1}\|\varphi\|_{S(T),p}.
$$
Then we use the estimate (\ref{aaa.58-new}), with $K=S(T)$, to
obtain
$$
\|\nabla \varphi\|_{T,p}\leq Ch_T^{-1}\|\varphi\|_{S,p}\leq
Ch_T^{-1}\|\varphi\|_{T,p},
$$
where $S$ is a ball inside of $T$ with a diameter proportional to
$h_T$. This completes the proof of the lemma.
\end{proof}

\begin{lemma}\label{appendix.inverse-inq-2}
Let $\T_h$ be a finite element partition of $\Omega$ consisting of
polygons or polyhedra. Assume that $\T_h$ satisfies all the
assumptions {\bf A1-A4} and $p\ge r \ge 1$ be any two real numbers.
Then, there exists a constant $C=C(n)$ such that
\begin{equation}\label{aaa.189}
\|\varphi\|_{L^p(\Omega)} \le Ch^{\frac{d}{p}-\frac{d}{r}}
\|\varphi\|_{L^r(\Omega)}
\end{equation}
for any piecewise polynomial $\varphi$ of degree $n$ on $\T_h$.
\end{lemma}
\begin{proof}
For any $T\in \T_h$, let $S(T)$ be the circumscribed simplex that is
shape regular. It follows from the standard inverse inequality that
$$
\|\varphi\|_{T,p}\leq \|\varphi\|_{S(T),p}\leq
Ch_T^{\frac{d}{p}-\frac{d}{r}}\|\varphi\|_{S(T),r}.
$$
We then use the estimate (\ref{aaa.58-new}), with $K=S(T)$, to
obtain
\begin{eqnarray}
\|\varphi\|_{T,p}&\leq & Ch_T^{\frac{d}{p}-\frac{d}{r}}
\|\varphi\|_{S,r}\nonumber\\
&\leq &
Ch_T^{\frac{d}{p}-\frac{d}{r}}\|\varphi\|_{T,r},\label{App:local-inverse}
\end{eqnarray}
where $S$ is a ball inside of $T$ with a diameter proportional to
$h_T$. It follows from (\ref{App:local-inverse}) that
\begin{eqnarray*}
\|\varphi\|_{L^p(\Omega)}^p &=&
\sum_{T\in\T_h}\|\varphi\|_{T,p}^p\\
&\leq & Ch^{d-\frac{pd}{r}} \sum_{T\in\T_h} \|\varphi\|_{T,r}^p.
\end{eqnarray*}
Since $\frac{r}{p}\le 1$, then we have from the above inequality
that
\begin{eqnarray*}
\|\varphi\|_{L^p(\Omega)}^r &\leq & Ch^{\frac{rd}{p}-d}
\left(\sum_{T\in\T_h} \|\varphi\|_{T,r}^p\right)^{\frac{r}{p}}\\
&\leq & Ch^{\frac{rd}{p}-d} \sum_{T\in\T_h} \|\varphi\|_{T,r}^r,
\end{eqnarray*}
which implies the desired inverse inequality (\ref{aaa.189}).
\end{proof}

\vfill\eject

\end{document}